\documentclass[preprint,12pt]{elsarticle}

\usepackage{amssymb}
\usepackage{natbib}

\usepackage[latin1]{inputenc}
\usepackage{amsmath}
\usepackage{amsfonts}
\usepackage{amsthm}
\usepackage{graphicx}
\usepackage{xcolor}
\newtheorem{theorem}{Theorem}
\newtheorem{cor}[theorem]{Corollary}
\newtheorem{lem}[theorem]{Lemma}
\newtheorem{defn}[theorem]{Definition}
\newtheorem{theo}[theorem]{Theorem}
\date{}

\journal{ }
\begin{document}

\begin{frontmatter}



\title{Counting closed walks in infinite regular trees using Catalan and Borel's triangles}


\author[inst1]{Lord C. Kavi\corref{cor1}}
\ead{lkavi060@uottawa.ca}
\cortext[cor1]{Corresponding author}
\affiliation[inst1]{organization={Department of Mathematics, University of Ottawa},
            city={Ottawa},
            state={ON},
            country={Canada}}

\author[inst1]{Michael W. Newman}
\ead{mnewman@uottawa.ca}

\begin{abstract} 
We count the number of closed walks on a vertex in a regular tree using  the Catalan's triangle and also  the Borel's triangle, showing another combinatorial structure counted by these two array of numbers. 
\end{abstract}




\end{frontmatter}


\section{Introduction}
Let $G$ be an infinite $\delta$-regular tree. What is the number of closed walks of length $2n$, $n\in \mathbb{N}$ that starts and ends at vertex $v\in V(G)$? A a well-known result uses generating function (see \cite{rowland2009number} and all the references therein). It was shown that the generating function is
\[f_\delta(t)=\frac{2(\delta-1)}{\delta-2+\delta\sqrt{1-4(\delta-1)t^2}}.\] 
Our new result gives a combinatorial alternative approach.  We relate the number of closed walks to the Catalan's triangle and also the Borel's triangles. The Borel's triangle is an array of numbers that are closely related to the Catalan numbers and has recently appeared in several studies in commutative algebra, combinatorics and discrete geometry, Cambrian Hopf algebras \cite{chatel2017cambrian}, quantum physics \cite{lakshminarayan2014diagonal} and permutation patterns \cite{remmel2014consecutive}. Cai and Yan \cite{cai2019counting} studied some classes of objects that are counted by Borel's triangle and characterized
their combinatorial structures. We find no study that presents an application of Borel's triangle and Catalan's triangle in solving the well-known closed walk counting problem. We do so in this paper. 

\section{Preliminaries and setting up the problem}
Recall that in the Catalan's triangle, $C_{n,k}$ counts the number of lattice paths in the coordinate plane from $(0, 0)$ to $(n, k)$ that do not go above the line $y = x$. Explicitly,
	\begin{align*}
	    C_{n,k} = \frac{n-k+1}{n+1}\binom{n+k}{n}.
	\end{align*}
Catalan's triangles are the sequences $A009766$ on the On-line Encyclopedia of Integer Sequences (OEIS) \cite{sloane}.
The entries of $C_{n,k}$ for  values of $n$ and $k$ with $0\le n,k\le 7$, are listed below.
\[\begin{tabular}{|c|c|c|c|c|c|c|c|c|}
	\hline 
	n$\backslash$ k & 0 & 1 & 2 & 3 & 4 & 5 & 6 & 7 \\ 
	\hline 
	0 & 1 &  &  &  &  &  &  &  \\ 
	\hline 
	1 & 1 & 1 &  &  &  &  &  &  \\ 
	\hline 
	2 & 1 & 2 & 2 &  &  &  &  &  \\ 
	\hline 
	3 & 1 & 3 & 5 & 5 &  &  &  &  \\ 
	\hline 
	4 & 1 & 4 & 9 & 14 & 14 &  &  &  \\ 
	\hline 
	5 & 1 & 5 & 14 & 28 & 42 & 42 &  &  \\ 
	\hline 
	6 & 1 & 6 & 20 & 48 & 90 & 132 & 132  &  \\ 
	\hline 
	7 &  1 & 7 & 27 & 75 & 165 & 297 & 429 & 429\\
	\hline 
\end{tabular} \]

The Borel's triangle $\{B_{n,k} : 0 \le  k \le  n\}$ on the other hand, is an array of numbers obtained from an invertible transformation
to Catalan's triangle by the equation (see Cai and Yan \cite{cai2019counting})
\begin{align}
    B_{n,k}=\sum_{s=k}^{n}\binom{s}{k}C_{n,s}.\label{cata-borel}
\end{align}
Barry \cite{barry2013note} gave an explicit formula as
\begin{align*}
    B_{n,k}&=\frac{1}{n}\binom{2n+2}{n-k}\binom{n+k}{n}.
\end{align*}
Borel's triangles are the sequences $A234950$ on the On-line Encyclopedia of Integer Sequences (OEIS) \cite{sloane}. The entries of $B_{n,k}$ for  values of $n$ and $k$ with $0\le n,k\le 7$, are listed below.

\[\begin{tabular}{|c|c|c|c|c|c|c|c|c|}
	\hline 
	n$\backslash$ k & 0 & 1 & 2 & 3 & 4 & 5 & 6 & 7 \\ 
	\hline 
	0 & 1 &  &  &  &  &  &  &  \\ 
	\hline 
	1 & 2 & 1 &  &  &  &  &  &  \\ 
	\hline 
	2 & 5 & 6 & 2 &  &  &  &  &  \\ 
	\hline 
	3 & 14 & 28 & 20 & 5 &  &  &  &  \\ 
	\hline 
	4 & 42 & 120 & 135 & 70 & 14 &  &  &  \\ 
	\hline 
	5 & 132 & 495 & 770 & 616 & 252 & 42 &  &  \\ 
	\hline 
	6 & 429 & 2002 & 4004 & 4368 & 2730 & 924 & 132 &  \\ 
	\hline 
	7 &  1430 &  8008 &  19656 &  27300 & 23100 &  11880 & 3432 & 429\\
	\hline 
\end{tabular} \]

	We now set up the problem of interest. Let $G$ be an infinite $\delta$-regular tree.   A finite $\delta$-regular graph of order $m$ with girth greater than $2n$, $n\in \mathbb{N}$ acts as a tree locally, so the results in this article apply to such graphs as well. To find all such closed walks, we suppose $G$ is rooted at $v$.  Any closed walk from the root $v$ can be described uniquely as a sequence of moves away from the root ($R$) and towards the root ($L$). Call such a sequence an $RL$-sequence.
	\begin{defn} An $RL$- sequence is said to be {\em balanced} if there are as many $R$'s as $L$'s.
	\end{defn}
	Hence, there are no odd closed walks. 
	\begin{defn}
		A balanced $RL$- sequence is said to be {\em legal} if it has at most as many $L$'s as $R$'s at any point in the sequence.
	\end{defn}
	Thus, a closed walk from the root $v$ is a balanced legal $RL$-sequence. A balanced legal $RL$-sequence of length $2n$ can be considered as a Dyck path of length $2n$ (or semi-length $n$).
	\begin{defn}
		A {\em component} of an $RL$-sequence $S$ is formed when the sequence touches the root vertex $v$.  The first component  starts from the first $R$ from $v$ to the first $L$ that touches $v$. The second component  starts from the second $R$ move from $v$ to the second $L$ to that touches $v$, and so on.
	\end{defn} 
	The following is then immediate.
	\begin{lem}\label{decom}
		Every balanced legal $RL$-sequence is a sequence of its components.\qed
	\end{lem}

	\section{Main Results}
		From henceforth, a sequence shall mean a balanced legal sequence.
	Let $\mathcal{S}_{n,k}$ be the set of $RL$-sequences of length $2n$ with $k$ components, and $S_{n,k}=|\mathcal{S}_{n,k}|.$
	\begin{lem}
		\label{lemm1}
		Let $n,k\in \mathbb{Z}^+, k\le n.$ The number of  sequences of length $2n$ with exactly $k$ components  is equal to the number of sequences of length $2(n-1)$ with at least $k-1$ components. That is, 
		\[S_{n,k}=\sum_{j=k-1}^{n-1}S_{n-1,j}.\]
	\end{lem}
	\begin{proof} We define a deletion function $f_{i}$. The deletion function $f_{i}$ removes a pair $RL$ that forms the initial$(R)$ and terminal$(L)$ letters of the $i$th component of a sequence. 
		Let $\omega \in \mathcal{S}_{n,k}$. A sequence $\alpha\in \mathcal{S}_{n-1,j},$ for $ k-1\leq j \leq n-1$ can be achieved by applying a deletion function $f_{i}$ to $\omega$. 
		
		We note the following observations.
		\begin{itemize}
			\item[i.] $f_{i}(\omega) \in \mathcal{S}_{n-1,k-1}$ if the $i$th component consists of only $RL$. Otherwise:
			\item[ii.] $f_{i}(\omega) \in \mathcal{S}_{n-1,j}$, $k\leq j \leq n-1$.
		\end{itemize}
		We show that for fixed $i$, $f_{i}$ is injective.\\
		Suppose  $\omega_1, \omega_2 \in \mathcal{S}_{n,k}$ and $\omega_1 =A_1A_2\dots A_{k}\neq B_1B_2\dots B_{k}=\omega_2$, where $A_j$ and $B_j$  are components for all $j\in [1,k]$, but $f_{i}(\omega_1)=f_{i}(\omega_2)$. Then, we have  $$f_{i}(\omega_1)=A_1A_2\dots A_{i-1}\bar{A}_iA_{i+1}\dots A_{k}$$ and $$f_{i}(\omega_2)=B_1B_2\dots B_{i-1}\bar{B}_iB_{i+1}\dots B_{k},$$ where $\bar{A}_i$ and $\bar{B}_i$ are some legal sequences. So $f_{i}(\omega_1)=f_{i}(\omega_2)$ implies $A_j=B_j \quad \forall j\in [1,k]\backslash \{i\}$ and $\bar{A}_i=\bar{B}_i$. We note that the deleted terms of component $i$ in each of $\omega_1$ and $\omega_2$  are $R$ and $L$. Thus, $A_i=R\bar{A}_iL$ and $B_i=R\bar{B}_iL$. But since $\bar{A}_i=\bar{B}_i$, then $A_i=B_i$ which necessarily implies $\omega_1=\omega_2$. Thus, proving injectivity of $f_i$ for a fixed $i$. 
		\\
		\noindent
		Now, we show for $i$ fixed, $f_{i}$ is surjective.\\
		Given $\alpha\in \mathcal{S}_{n-1,j}, k-1\leq j \leq n-1$, we construct an $\omega\in \mathcal{S}_{n,k}$ as follows:\\ First we decompose $\alpha$ into its components, say $\alpha=C_1C_2C_3\dots C_j$. Now consider the components from $i$ through to the $(j-k+i)$th component of $\alpha$, that is $C_iC_{i+1}\dots C_{j-k+i}$, call it $\varphi$. Now place $R$ and $L$ in front and behind $\varphi$ respectively and call it $C_i^*$. Thus $C_i^*= RC_iC_{i+1}\dots C_{j-k+i}L$. Note that $C_i^*$ is a single component.  Then set $\omega= C_1C_2\dots C_{i-1}C_i^* \underbrace{C_{(j-k+i+1)}\dots C_{j}}_{k-i}$.
		Thus for every $\alpha=C_1C_2C_3\dots C_j$, there exists $$ \omega =C_1C_2\dots C_{i-1}C_i^* \underbrace{C_{(j-k+i+1)}\dots C_{j}}_{k-i}$$ such that $$f_{i}(\alpha)=f_{i}(C_1C_2\dots C_{i-1}C_i^* C_{(j-k+i+1)}\dots C_{j})=C_1C_2C_3\dots C_j.$$	Hence there is a bijection $f_{i}:\mathcal{S}_{n,k}\rightarrow \cup_{j=k-1}^{n-1}\mathcal{S}_{n-1,j}$, which then implies the claim.
	\end{proof}
	Thus the number of balanced sequences of length $2n$ with $k$ components is the  sum of the number of balanced sequences of length $2(n-1)$ with at least $k-1$ components.

	Using Lemma~\ref{lemm1}, we have the following theorem.	
	\begin{theo}\label{closedwalks}
		Let $G$ be an infinite $\delta$-regular tree.  The number of closed walks of length $2n$ at a vertex $v$ of $G$ is
		\begin{eqnarray}
			W_{2n} =
			\sum_{k=1}^n\left[\delta^{k}(\delta-1)^{n-k} \sum_{j \geq k-1}S_{n-1,j}\right]. \label{maineqn}
		\end{eqnarray}
	\end{theo}
	\begin{proof} By Lemma~\ref{decom}, the closed walks of length $2n$ can be decomposed into balanced legal sequences of the various number of components, $k=1,\dots,n$. In a sequence, an $R$ move starting at $v$ has $\delta$ possibilities while an $R$ move at any other vertex has $\delta-1$ possibilities but an $L$ move is completely determined since $G$ is a tree. Hence a  sequence with $k$ components has $\delta^{k} (\delta-1)^{n-k}$ possibilities. Hence by Lemma~\ref{lemm1}, there are $\delta^{k} (\delta-1)^{n-k}\sum_{j \geq k-1}S_{n-1,j}$ such sequences with $k$ components. But $k$ runs from $1$ through $n$, so, we have the desired result.
		
	\end{proof}
	\begin{cor}
		Let $G$ be a finite $\delta$-regular graph of order $m$. Suppose $G$ has girth greater than $2n\in \mathbb{Z}$. Then the number of closed walks of length $2n$ at a vertex $v$ in $G$ is $W_{2n}$ as in Equation~(\ref{maineqn}).
		\qed
	\end{cor}
	\noindent
	We note that for $n>0$, $ \sum_{j \geq 0}S_{n-1,j}= \sum_{j \geq 1}S_{n-1,j}=C_{n-1}$, the $(n-1)$th Catalan number and so the $n$th Catalan number, $C_n$ is the sum of number of balanced sequences of length $2n$ with at least $1$ component. We summarize this in the corollary that follows. 
	\begin{cor} The $n$th Catalan number, $C_n$, for $n>0$, is given by
		\begin{eqnarray}\label{cn}
		C_n&=& \sum_{j = 1}^{n}S_{n,j}=  \sum_{k=1}^{n} \sum_{j \geq k-1}S_{n-1,j}.
		\end{eqnarray} \qed
	\end{cor}
The second equality in Equation~(\ref{cn}) comes  directly from using Lemma~\ref{lemm1}.	
	
	The following result by Lubotzky et al. \cite{lubotzky1988ramanujan} follows as a consequence of Theorem~\ref{closedwalks}. See also \cite{dsouza2013combinatorial}.
	\begin{cor}
		Let $G$ be an infinite $\delta$-regular tree.	The number of walks of length $2n$ in $G$ that start at $v$ and end at $v$ for the first time is
		\begin{align*}
		W_{2n} &=\delta(\delta-1)^{n-1} \sum_{j \geq 0}S_{n-1,j}\\
		&=\delta(\delta-1)^{n-1}C_{n-1}.
		\end{align*}
		\begin{proof}The result follows from the fact that such a walk contains just one component, $k=1$.
		\end{proof}
	\end{cor}
	We can get similar result if we seek closed walks that touch the vertex exactly twice, that is, we have exactly two components.
		\begin{cor}
		Let $G$ be an infinite $\delta$-regular tree.	The number of walks of length $2n$ in $G$ that start at $v$ and end at $v$ after touching it the second time is
		\begin{align*}
			W_{2n} &=\delta^2(\delta-1)^{n-2}C_{n-1}.
		\end{align*}
		\begin{proof}The result follows from the fact that such a walk contains two components, $k=2$. And using the fact that $ \sum_{j \geq 0}S_{n-1,j}= \sum_{j \geq 1}S_{n-1,j}=C_{n-1}$ yields the desired result. 
		\end{proof}
	\end{cor}
	
	We can say a bit more. The following result is due to Cai and Yan \cite{cai2019counting}.
	
	\begin{theo}[\cite{cai2019counting}]The entry $C_{n,k}$ of Catalan's triangle counts  
	Dyck paths of semi-length $n + 1$ that have $k$ up-steps (or down-steps) not at ground level. Equivalently, it is the set of Dyck paths of semi-length $n+1$ with $n+1-k$ returns to the $x$-axis (not counting the starting point $(0, 0)$).
	    
	\end{theo}
 Thus, $C_{n,k}$ counts the RL sequences of length $2(n+1)$ with $n+1-k$ components. 

We have then that, $C_{n-1,n-k}$ counts the  RL sequences of length $2n$ with $k$ components. Thus, there are $\delta^k(\delta-1)^{n-k}C_{n-1,n-k}$ closed walks of length $2n$ with $k$ components (or that returns to vertex $v$ exactly $k$ times). But since $k$ runs from $1$ through to $n$, we have the following result which gives the number of closed walks in terms of the Catalan's triangles. 
	
	\begin{theo}\label{walkscatalan}
		Let $G$ be an infinite $\delta$-regular tree (or a finite $\delta$-regular graph of order $m$ with girth greater than $2n$).  The number of closed walks of length $2n$ at a vertex $v$ of $G$ is
		\begin{eqnarray}
			W_{2n} =
			\sum_{k=1}^n\delta^{k}(\delta-1)^{n-k}C_{n-1,n-k}, \label{eqncat}
		\end{eqnarray}
		where $C_{n,k}$ is the Catalan's triangle. \qed
	\end{theo}
	
	Now comparing Theorem~\ref{closedwalks} and Theorem~\ref{walkscatalan}, we can deduce another combinatorial interpretation of the $(n-1,n-k)$ entry of the Catalan's triangle. 
	\begin{cor}\label{lemma2} In the Catalan's triangle,
	$C_{n-1,n-k} $ counts the number of
 $RL$ sequences of length $2(n-1)$ with at least $k-1$ components. Equivalently, 
	    it counts Dyck paths of semi-length $n-1$ with at least $k-1$ returns to the $x$-axis (not counting the starting point (0,0)).
	
	That is, 
	    \begin{align*}
	    C_{n-1,n-k} =  \sum_{j \geq k-1}S_{n-1,j}.
	\end{align*}
	\end{cor}
	
Now, recall from Equation~\ref{cata-borel}, we have
 \begin{align*}
    B_{n,k}=\sum_{s=k}^{n}\binom{s}{k}C_{n,s}.
\end{align*}
Thus, we can express the number of closed walks at a vertex in terms of Borel's triangle as well.

\begin{theo}\label{closedwalksBorel}
	Let $G$ be an infinite $\delta$-regular tree, (or a finite $\delta$-regular graph of order $m$ with girth greater than $2n$).  The number of closed walks of length $2n$ at a vertex $v$ of $G$ is
	\begin{eqnarray*}
		W_{2n} &=& \sum_{\ell=1}^{n}(-1)^{n-\ell} B_{n-1,n-\ell}\,\delta^{\ell},\\
		&=& \sum_{\ell=0}^{n-1}(-1)^\ell B_{n-1,\ell}\,\delta^{n-\ell},
	\end{eqnarray*} where $B_{n,k}$ is Borel's triangle.
\end{theo}

\begin{proof}
Consider the coefficient of $\delta^\ell$ in Equation~(\ref{eqncat}). That is,
\
\begin{eqnarray*}
    [\delta^\ell]W_{2n}&=&[\delta^\ell]\sum_{k=1}^n\delta^{k}(\delta-1)^{n-k}C_{n-1,n-k}\\
    &=& [\delta^\ell]\sum_{k=1}^\ell \delta^{k}(\delta-1)^{n-k}C_{n-1,n-k}\\
    &=& [\delta^{\ell-k}]\sum_{k=1}^\ell (\delta-1)^{n-k}C_{n-1,n-k}\\
    &=& [\delta^{\ell-k}]\sum_{k=1}^\ell \sum_{i=0}^{n-k}\binom{n-k}{i}\delta^{n-k-i}(-1)^i C_{n-1,n-k}\\
    &=& \sum_{k=1}^\ell \binom{n-k}{n-\ell}(-1)^{n-\ell} C_{n-1,n-k}\\
    &=& (-1)^{n-\ell}\sum_{k=1}^\ell \binom{n-k}{n-\ell} C_{n-1,n-k}\\
    &=& (-1)^{n-\ell}\sum_{s=n-\ell}^{n-1}\binom{s}{n-\ell} C_{n-1,s}\\
    &=& (-1)^{n-\ell} B_{n-1,n-\ell}.
\end{eqnarray*}
Now, since $\ell$ runs from $1$ through to $n$, we have 
\begin{align*}
    W_{2n}= \sum_{l=1}^{n} (-1)^{n-\ell} B_{n-1,n-\ell}\delta^\ell.
\end{align*}
\end{proof}	
\subsection{Examples}
We end this note with the following examples. Let $G$ be a $\delta$-regular infinite tree (or a finite $\delta$-regular graph of order $m$ with girth greater than $2n$). Then the number of closed walks of length $2n$ centred at a vertex $v\in G$ for $n=1,2,\dots,6$ are given in the table below.
\newline	
\fontsize{11}{11}\selectfont

	\begin{tabular}{|c|c|} 
		\hline 
		\multicolumn{2}{|c|}{For length $2$} \\ 
		\hline 
		& $\delta\times \mathbf{1}$  \\  
		\hline 
		$W_2=$ & $\delta$ \\
			\hline 
		
		    \hline
		\multicolumn{2}{|c|}{For length $4$} \\ 
		\hline 
		& $\delta^2\times \mathbf{1}$  \\ 
		\hline 
		& $\delta(\delta-1)\times \mathbf{1}$ \\ 
		\hline 
		$W_4=$ & $\mathbf{2\delta^2-\delta}$ \\ 
			\hline 
		
		    \hline 
		\multicolumn{2}{|c|}{For length 6} \\ 
		\hline 
		& $\delta^3\times \mathbf{1}$ \\ 
		\hline 
		& $\delta^2(\delta-1)\times \mathbf{2}$ \\ 
		\hline 
		& $\delta(\delta-1)^2\times \mathbf{2}$ \\ 
	\hline 
		$W_6=$ & $\mathbf{5\delta^3-6\delta^2+2\delta}$ \\ 
			\hline 
		
		    \hline
		\multicolumn{2}{|c|}{For length 8} \\ 
		\hline 
		& $\delta^4 \times \mathbf{1} $\\ 
		\hline 
		& $\delta^3(\delta-1)\times \mathbf{3}$ \\ 
		\hline 
		& $\delta^2(\delta-1)^2\times \mathbf{5} $\\ 
		\hline 
		& $\delta(\delta-1)^3\times \mathbf{5}$\\ 
		\hline
		$W_8=$ & $\mathbf{14\delta^4-28\delta^3+20\delta^2-5\delta}$ \\ 
		\hline 
			 
			\hline 
		\multicolumn{2}{|c|}{For length 10} \\ 
		\hline 
		& $\delta^5\times \mathbf{1}$ \\ 
		\hline 
		&$ \delta^4(\delta-1)\times \mathbf{4}$ \\ 
		\hline 
		& $\delta^3(\delta-1)^2\times \mathbf{9}$ \\ 
		\hline 
		& $\delta^2(\delta-1)^3\times \mathbf{14}$ \\ 
		\hline 
		& $\delta(\delta-1)^4\times \mathbf{14} $\\ 
		\hline 
		$W_{10}=$ & $\mathbf{42\delta^5-120\delta^4+135\delta^3-70\delta^2+14\delta}$ \\ 
		\hline 
		
		    \hline
		\multicolumn{2}{|c|}{For length 12} \\ 
		\hline 
		& $\delta^6\times \mathbf{1}$ \\ 
		\hline 
		& $\delta^5(\delta-1)\times \mathbf{5}$ \\ 
		\hline 
		& $\delta^4(\delta-1)^2\times \mathbf{14}$ \\ 
		\hline 
		& $\delta^3(\delta-1)^3\times \mathbf{28}$ \\ 
		\hline 
		&$ \delta^2(\delta-1)^4\times\mathbf{ 42}$ \\ 
		\hline 
		& $\delta(\delta-1)^5\times \mathbf{42} $\\ 
		\hline 
		$W_{12}=$ & $\mathbf{132\delta^6- 495\delta^5+770\delta^4-616\delta^3+252\delta^2-42\delta}$ \\ 
		\hline 		
	\end{tabular}
\normalsize

\section{Acknowledgement}
Research is supported by NSERC Canada.

\bibliographystyle{plain}
\bibliography{references}



\end{document}